\documentclass[a4paper]{amsart}
\usepackage{amssymb,latexsym}
\usepackage[utf8x]{inputenc}
\usepackage{amsmath,amsthm}
\usepackage{amsfonts,mathrsfs,url}
\usepackage{cleveref}
\usepackage{enumerate,units}
\usepackage[all]{xy}
\usepackage{graphicx}

\theoremstyle{plain}
\newtheorem{theorem}{Theorem}[section]
\newtheorem{corollary}[theorem]{Corollary}
\newtheorem{lemma}[theorem]{Lemma}
\newtheorem{proposition}[theorem]{Proposition}

\newtheorem{fact}[theorem]{Fact}
\newtheorem*{claim}{Claim}
\newtheorem*{theorem*}{Theorem}
\newtheorem{introtheorem}{Theorem}

\theoremstyle{definition}
\newtheorem{definition}[theorem]{Definition}
\newtheorem{example}[theorem]{Example}

\theoremstyle{remark}
\newtheorem*{remark}{Remark}

\numberwithin{equation}{section}
\newcommand{\forkindep}[1][]{%
  \mathrel{
    \mathop{
      \vcenter{
        \hbox{\oalign{\noalign{\kern-.3ex}\hfil$\vert$\hfil\cr
              \noalign{\kern-.7ex}
              $\smile$\cr\noalign{\kern-.3ex}}}
      }
    }\displaylimits_{#1}
  }
}

\newenvironment{claimproof}[1][\proofname]
  {%
    \proof[#1]%
      \renewcommand*\qedsymbol{‌​$\square$ (claim)}%
  }
  {%
    \endproof%
  }

\newcounter{step}                   
    {\hfill $\clubsuit$             
     \vspace{7pt}\par}
     
\newcommand{\GDp}{{\Gamma/{\Delta_p}}}
\newcommand{\GD}{{\Gamma/{\Delta}}}
\newcommand{\Deltap}{{\Delta_p}}

\title[Field-QE for Strongly Dependent Field]{Eliminating Field Quantifiers in Strongly Dependent Henselian Fields}

\date{\today}
\author[Y. Halevi]{Yatir Halevi$^*$}
\thanks{$^*$The research leading to these results has received funding from the European Research Council under the European Union’s Seventh Framework Programme (FP7/2007-2013)/ERC Grant Agreement No. 291111.}
\address{$^*$Einstein Institute of Mathematics\\
	The Hebrew University of Jerusalem\\
	Givat Ram\\
	Jerusalem 91904\\
	Israel\\}
\email{yatir.halevi@mail.huji.ac.il}
\urladdr{http://ma.huji.ac.il/\textasciitilde yatirh/}

\author[A. Hasson]{Assaf Hasson$^\dagger$}
\thanks{$^\dagger$ Supported by ISF grant No. 181/16}
\address{$^\dagger$Department of mathematics\\
	Ben Gurion University of the Negev\\
	Be'er Sehva\\
	Israel} \email{hassonas@math.bgu.ac.il} \urladdr{http://www.math.bgu.ac.il/\textasciitilde hasson/}

\date{\today}

\begin{document}
	\begin{abstract}
		We prove elimination of field quantifiers for strongly dependent henselian fields in the Denef-Pas language. This is achieved by proving the result for a class of fields generalizing algebraically maximal Kaplansky fields. We deduce that if $(K,v)$ is strongly dependent then so is its henselization.  
	\end{abstract}
	
	\maketitle
	\section{Introduction}
	This paper stemmed from the need for a complete proof that algebraically maximal Kaplansky fields eliminate field quantifiers (in the sense\footnote{See \cite[Appandix A]{silvain} for a more detailed discussion.}, e.g., of \cite[Definition 1.14]{ShDep09}). It quickly became clear that the same methods could be applied to prove elimination of field quantifiers for all strongly dependent henselian fields. 
	
	While elimination of field quantifiers for algebraically maximal Kaplansky fields may be folklore, we could not find a proof in the literature, though several closely related theorems do exist. In \cite[Theorem 2.6]{amc-kuh} Kuhlmann proves that such valued fields admit quantifier elimination relative to a structure he calls an \emph{amc-structure of level $0$}. It is well known that this structure is essentially the RV-structure (see for instance \cite[Section 3.2]{flenner}). In this language, Kuhlmann proves that if $L$ and $F$ are models of a theory of an algebraically maximal Kaplansky field and $K$ is a common substructure then \[RV_L\equiv_{RV_K}RV_F \Longrightarrow (L,v)\equiv_{(K,v)} (F,v),\] where the valued fields are considered, for instance, in the $L_{div}$ language. This is proved by showing that every embedding $RV_L\hookrightarrow RV_F$ (over $RV_K$) lifts to an embedding $(L,v)\hookrightarrow (F,v)$ (over $(K,v)$), provided that $F$ is $|L|^+$-saturated,
	
	Using this result B\'elair proves that, in equi-characteristic $(p,p)$, every algebraically maximal Kaplansky field eliminates field quantifiers in the Denef-Pas language, the 3-sorted language enriched with an angular component map (see \cite[Lemma 4.3]{belair}). It seems, though it is not claimed, that B\'elair's proof may apply to the mixed characteristic case as well.  
	
	Using ideas from \cite[Chapter 3]{johnson}, this result may be extended to strongly dependent henselian valued fields. The main results of this paper are the following
	\begin{introtheorem}\label{T:theorem A}
		Let $(K,v)$ be a henselian valued field, admitting angular component maps, and such that it is either
		\begin{enumerate}
			\item p-valued of rank $d$, 
			\item algebraically maximal Kaplansky or
			\item strongly dependent
		\end{enumerate}
		then $(K,v)$ eliminates field quantifiers in the (generalized) Denef-Pas language.
	\end{introtheorem}
	
	Shelah's conjecture (\cite{ShDep09}), usually interpreted as stating that strongly dependent fields which are neither real closed nor algebraically closed are henselian, is our main motivation for carrying out the present research. Strong dependence will, however, be used as a black box, and will never be invoked explicitly. For a more detailed discussion of strongly dependent henselian fields the reader is referred to \cite{HaH2017} and references therein. 
	
	As a consequence of our main result we deduce a transfer principle, providing a new method for constructing strongly dependent fields:
	\begin{introtheorem}\label{T:theorem B}
		Let $(K,v)$ be a strongly dependent valued field. Then its henselization $(K^h,v)$ is also strongly dependent. If, in addition, $(K,v)$ is Kaplansky then also its inertia field $(K^t,v)$ is strongly dependent. 
	\end{introtheorem}
	
	In Section \ref{s:geometric}  we show that any field of finite dp-rank admitting a non-trivial henselian valuation is \emph{geometric} in the sense of \cite{HrP1994}, and that if it also has geometric elimination of imaginaries it is either algebraically closed or real closed.
	
	In Appendix \ref{A:appendix} we show that our proof gives elimination of field quantifiers for strongly dependent henselian fields in the RV-language. 
	
	\section{preliminaries}
	\subsection{Valued Fields}
	We review some terminology and definitions. For a valued field $(K,v)$ let $vK$ denote the value group, $Kv$ the residue field, $res$ the residue map, $\mathcal{O}_K$ (or $\mathcal O$, if the context is clear) the valuation ring and $\mathcal{M}_K$ (or $\mathcal{M}$) its maximal ideal.
	
	Valued fields will be considered in the $3$-sorted language, with sorts for the base field, the value group and the residue field, with the obvious functions and relations. An n-th \emph{angular component map} on a valued field $(K,v)$ is a multiplicative group homomorphism \[ac_n:K^\times \to (K/\mathcal{M}^n)^\times\] such that $ac_n(a)=res_n(a)$ whenever $v(a)=0$, where $res_n:\mathcal{O}\to \mathcal{O}/\mathcal{M}^n$ is the projection on the n-th residue ring, we extend it to $ac_n:K\to K/\mathcal{M}^n$ by setting $ac_n(0)=0$.
	
	\begin{fact}\cite[Corollary 1.6]{pas90}\label{F:exist-ac}
		Every valued field $(K,v)$ has an elementary extension with an n-th angular component map on it.
	\end{fact}
	In fact, to obtain an n-th angular component $\aleph_1$-saturation of the field suffices. An ac-valued field is a valued field equipped with an angular component map (i.e. $ac_1$). An ac$_\omega$-valued field is a valued field equipped with an n-th angular component map for every $n$ such that $res_{m,n}\circ ac_n=ac_m$ for $n\geq m$, where $res_{m,n}: \mathcal{O}/\mathcal{M}^n\to \mathcal{O}/\mathcal{M}^m$ is the natural projection. We define two languages:  
	
	The $3$-sorted language of valued fields augmented  by a function symbol \emph{ac} for the angular component map, is called the Denef-Pas language. 
	
	The $3$-sorted language of valued fields augmented by new sorts for the different n-th residue rings, with the natural projections, the different n-th angular component maps and a compatible system of constants will be called the generalized Denef-Pas language. The structures we will be dealing with in this language will have discrete value group, and the constants will be interpreted as images in $\mathcal{O}/\mathcal{M}^n$ of an element of minimal valuation, see \cite[Section 4]{raf} or \cite{raf2} for more information.
	
	A rough characterization of strongly dependent henselian fields was given in \cite[Theorem 4.3.1]{johnson} and a little more explicitly in \cite[Theorem 5.14]{HaH2017}. We make it explicit here, but we first remind some necessary definitions: 
	
	\begin{definition}
		\begin{enumerate}
			\item A valued field $(K,v)$ of residue characteristic $p>0$ is a \emph{Kaplansky field} if the value group is $p$-divisible, the residue field is perfect and does not admit any finite separable extensions of degree divisible by $p$.
			\item A valued field $(K,v)$ of residue characteristic $0$ is Kaplansky. 
			\item $(K,v)$ is \emph{algebraically maximal} if if does not admit any immediate algebraic extension.
			\item $(K,v)$ is \emph{$p$-valued of $p$-rank $d$} if
			\begin{list}{•}{}
				\item $char(Kv)=p>0$ and $char(K)=0$ and
				\item $\dim_{\mathbb{F}_p} \mathcal{O}_K/(p)=d$.
			\end{list}
		\end{enumerate}

	\end{definition}

	Let $(K,v,\Gamma, k)$ be a valued field with value group $\Gamma$, residue field $k$, valuation ring $\mathcal{O}$ and maximal ideal $\mathcal{M}$. Given a convex subgroup $\Delta\le \Gamma$, we let $v^{\Gamma/\Delta}:K\to \Gamma/\Delta$ denote the coarsening of $v$ with valuation ring $\mathcal{O}^{\Gamma/\Delta}=\{x\in K: \exists \delta\in \Delta,\, v(x)>\delta\}$, maximal ideal $\mathcal{M}^{\Gamma/\Delta}=\{x\in K:v(x)>\Delta\}$ and residue field $k^{\Gamma/\Delta}$. 
	
	Setting  $K_1:=k^{\Gamma/\Delta}$
	we let  $v^\Delta:K_1\to \Delta$ denote the valuation given by \[v^\Delta (a+\mathcal{M}^{\GD})=
	\begin{cases}
	v(a) & \text{if } a\in\mathcal{O}^\GD\setminus \mathcal{M}^\GD \\
	\infty & \text{otherwise.}
	\end{cases}\]
	It has valuation ring $\mathcal{O}^\Delta=\{a+\mathcal{M}^\GD: v(a)\geq 0\}$ with maximal ideal $\mathcal{M}^\Delta=\{a+\mathcal{M}^\GD:v(a)>0\}$ and residue field $k^{\Delta}$. 
	It is well known (and easy to check) that: 
	
	\begin{fact}
		The map $k\to k^\Delta$ given by $a+\mathcal{M}\mapsto (a+\mathcal{M}^\GD)+\mathcal{M}^\Delta$ is a field isomorphism.
	\end{fact}
	%
	%
	%
	%
	%
	
	If $(K,v)$ is a valued field of mixed characteristic $(0,p)$ the \emph{core field} of $K$ is the valued field $(K_1,v^{\Delta_p})$, where $\Delta_p$ is the minimal convex subgroup containing $v(p)$. It is also a valued field of mixed characteristic and if $(K,v)$ is henselian then so is the core field. Notice that $(K,v^{\Gamma/\Delta_p})$ is of equi-characteristic $(0,0)$.
	
	\begin{fact}\cite[Theorem 4.3.1]{johnson}\cite[Theorem 5.13]{HaH2017}\label{F:char-strongdep}
		If $(K,v)$ is a henselian valued field with $K$ strongly dependent then
		\begin{list}{•}{}
			\item if $(K,v)$ is of equi-characteristic $(p,p)$ then $(K,v)$ is an algebraically maximal Kaplansky field,
			\item if $(K,v)$ is of mixed characteristic $(0,p)$ then 
			\begin{enumerate}
				\item if $Kv$ is infinite then the core field of $(K,v)$ is algebraically maximal Kaplansky,
				\item if $Kv$ is finite then $(K,v)$ is a p-valued field (the core field is a p-adically closed field).
			\end{enumerate}
		\end{list}
	\end{fact}
	
	\subsection{Two theories of valued fields}\label{ss:2theories}
	Although our main goal is to show elimination of field quantifiers in the (generalized) Denef-Pas language for strongly dependent henselain fields, we actually show a bit more. We show it for the two theories given below, whose union encompasses (by Fact \ref{F:char-strongdep}) all strongly dependent fields. The first theory we consider is a generalization of a theory given in \cite[Section 3.2]{johnson}. The second theory is that of p-valued fields.
	
	\subsubsection{Strongly dependent henselian fields with an infinite residue field}
	We borrow the following terminology from \cite{johnson}: 
	
	\begin{definition}
		A valuation $v:K\to \Gamma$ is \emph{roughly $p$-divisible} if $[-v(p),v(p)]\subseteq p\Gamma$, where
		\[[-v(p),v(p)]=\begin{cases}
		\{0\} & \text{in pure characteristic 0}\\
		\Gamma & \text{in pure characteristic p}\\
		[-v(p),v(p)] & \text{in mixed characteristic}
		\end{cases}.\]
	\end{definition}
	
	\begin{remark}
		In the mixed characteristic case, if $vK$ is roughly $p$-divisible and $\Deltap$ is the minimal convex subgroup containing $v(p)$, the coarsening $(K,v^\GDp)$ is of equi-characteristic $0$ and $v^\Deltap$, the induced valuation  $Kv^\GDp$ then $(Kv^\GDp,v^\Deltap)$, is of mixed characteristic $(0,p)$ with a $p$-divisible value group.
	\end{remark}

	\begin{definition}
		Let $T_1$ be the theory of valued fields stating:
		\begin{list}{•}{}
			\item the valued field is henselian and defectless,
			\item the base field and the residue field are perfect,
			\item the valuation is roughly $p$-divisible,
			\item every finite field extension of the residue field has degree prime to $p$.
		\end{list}
	\end{definition}
	
	\begin{remark}
		\begin{enumerate}
			\item Perfection of the residue field follows, in fact, from the requirement that every finite field extension of the residue field has degree prime to $p$.
			\item Perfection of the base fields also follows from the other axioms (see below). 
			\item Every algebraically maximal Kaplansky field is a model of $T_1$.
			\item If $(K,v)\models T_1$ and $vK$ is not $p$-divisible then, by perfection of $K$ (or rough $p$-divisibility), it is necessarily of mixed characteristic.
		\end{enumerate}
	\end{remark}
	
	\begin{lemma}\label{L:infinite KV}
		Let $(K,v)$ be a strongly dependent henselian field. If $Kv$ is infinite then $(K,v)\models T_1$. 
	\end{lemma}
	\begin{proof}
		By Fact \ref{F:char-strongdep}, if $\mathrm{char}(K)=\mathrm{char}Kv$ then $(K,v)$ is algebraically maximal Kaplansky, and the lemma follows from the above remark. So we are reduced to the case where $(K,v)$ is of mixed characteristic $(0,p)$. By \cite[Corollary 5.15]{HaH2017} $(K,v)$ is defectless. By Fact \ref{F:char-strongdep}, since $Kv$ is infinite, the core field of $(K,v)$ is algebraically maximal Kaplansky. In particular, the convex sub-group $\Delta_p\le vK$ generated by $v(p)$ is $p$-divisible, so $(K,v)$ is roughly $p$-divisible. Since $Kv$ is strongly dependent (e.g., \cite[Proposition 5.2]{HaH2017}) it has no finite extensions of degree divisible by $p$ (\cite{KaScWa}). 
	\end{proof}
	
	We collect a few results, essentially, due to Johnson  (\cite[Section 3.2]{johnson}). Johnson states these results under the stronger assumption that the residue field is algebraically closed.  We repeat the proofs, sometimes verbatim, only to emphasize that this requirement is inessential. We start with an immediate application of henselianity:
	
	\begin{fact}\cite[Remark 3.2.2]{johnson}\label{F:sep-closure}
		Let $(L,v)/(K,v)$ be an extension of valued fields. Suppose $(L,v)$ is henselian and $K$ is relatively separably closed in $L$. Then $Kv$ is relatively separably closed in $Lv$.
	\end{fact}
	
	
	The main properties of models of $T_1$ are collected in the next proposition: 
	
	\begin{proposition}\label{P:has-p-roots-and-more}\cite[Proposition 3.2.3]{johnson}
		Let $(F,v)\models T_1$,  of residue characteristic $p$. Then
		\begin{enumerate}
			\item If $vF$ is $p$-divisible then any finite field extension of $F$ has degree prime to $p$,
			\item $a\in F^p$ if and only if $v(a)\in p\cdot vF$,
			\item If $K$ is relatively algebraically closed in $F$, then $K\models T_1$.
		\end{enumerate}
	\end{proposition}
	\begin{proof}
		\begin{enumerate}
			\item Let $L/F$ be a finite extension. Since $F$ is henselian and defectless \[[L:F]=(vL:vF)[Lv:Fv],\]
			but $(vL:vF)$ is prime to $p$ since $vF$ is $p$-divisible and $[Lv:Fv]$ is prime to $p$ by assumption.
			\item In case $vF$ is $p$-divisible if $a\in F\setminus F^p$ then the polynomial $x^p-a$ is irreducible, contradicting $(1)$. 
			
			So we now assume that $vF$ is not $p$-divisible.  Hence, for $\Deltap$, the convex subgroup generated by $v(p)$, we get that $(Fv^\GDp,v^\Deltap)\models T_1$ and has a $p$-divisible value group. Fix some $a\in F$ such that $v(a)\in p\cdot vF$. By considering $a/b^p$ for $v(a)=pv(b)$, we reduce to the case where $v(a)=0$. Because $(F,v^\GDp)$ is henselian of residue characteristic $0$, and $v^\GDp(a)=0$ we know that $a\in F^p$ if and only if $res^\GDp(a)\in (Fv^\GDp)^p$. Since $\Gamma/\Deltap$ is $p$-divisible, by the previous paragraph $Fv^\GDp=(Fv^\GDp)^p$, with the desired conclusion. 
			
			\item We first assume that $F$ has a $p$-divisible value group and show that $(K,v)\models T_1$ and $vK$ is $p$-divisible. Since $F$ is henselian and perfect and $K$ is algebraically closed in $F$, also $K$ is henselian and perfect. So $F/K$ is regular, implying that $F$ and $K^{\text{alg}}$ are linearly disjoint over $K$. Thus $p$ does not divide the degree of any finite extension of $K$. Indeed, if $K(a)/K$ is a finite extension with degree divisible by $p$ then by linear disjointness so is $F(a)/F$. It follows that $(K,v)$ is defectless, $Kv$ is perfect and $vK$ is $p$-divisible. Since every finite extension of $Kv$ may be lifted to a finite extension $K$, $p$ does not divide the degree of any finite extension of $Kv$. 
			
			Assume now that $vF$ is not $p$-divisible. Let $\Deltap, v^\GDp$ and $v^\Deltap$ be as before. By Fact \ref{F:sep-closure}, $Kv^\GDp$ is relatively algebraically closed in $Fv^\GDp$. As $(Fv^\GDp,v^\Deltap)$ is a model of $T_1$ with $p$-divisible value group by what we have done above so is $(Kv^\GDp,v^\Deltap)$. Thus the place $K\to Kv$ decomposes into $K\to Kv^\GDp\to Kv$ each of them henselian, defectless and roughly $p$-divisible. Thus so is $K\to Kv$, i.e., $(K,v)\models T_1$.
		\end{enumerate}
	\end{proof}
	
	\begin{fact}\cite[Lemma 3.12]{amc-kuh}\label{F:final-extension}
		Let $L$ and $F$ be two algebraically maximal Kaplansky fields (and hence defectless) and $K$ a common henselian subfield. Assume that both $vL/vK$ and $vF/vK$ are $p$-torsion groups and both $Lv/Kv$ and $Fv/Kv$ are purely inseparable algebraic extensions. Then the relative algebraic closures of $K$ in $L$ and $F$ are isomorphic.
	\end{fact}
	
	In order to use the above fact in our setting we will need a result from ramification theory, see \cite[Section 5.2]{EnPr} for notation.
	\begin{lemma}\label{L:EnPr-Ramific}
		Let $(K,v)$ be a henselian valued field of residue characteristic $0$ and $(L,v)$ and $(F,v)$ two algebraic extensions with $vL=vK=vF$. If $Fv$ and $Lv$ are isomorphic over $Kv$ (as fields) then $F$ and $L$ are isomorphic over $K$ (as fields\footnote{Indeed, even as valued fields.}).
	\end{lemma}
	\begin{proof}
		Since $(K,v)$ is of residue characteristic $0$ and henselian, $Kv$ is perfect and $(K,v)$ is defectless. Since $vL=vK=vF$, by \cite[Theorem 5.2.9(1)]{EnPr}, $L,F\subseteq K^t$, the inertia field of $K^{\text{alg}}/K$. The result now follows from \cite[Theorem 5.2.7(2)]{EnPr}.
	\end{proof}
	
	The following is an adaptation of \cite[Lemma 3.2.4]{johnson}.
	\begin{lemma}\label{L:adap-johnson}
		Let $L,F\models T_1$ and $K$ a common valued subfield. Assume that $vL=vK=vF$ and that both $Lv/Kv$ and  $Fv/Kv$ are purely inseparable algebraic extensions. Then the relative algebraic closures of $K$ in $L$ and $F$ are isomorphic.
	\end{lemma}
	\begin{proof}
		We may replace $K$ with the perfection of its henselization, thus it is enough to show that the relative algebraic closures are isomorphic as fields over $K$.
		
		If $vK$ is $p$-divisible, then $L$ and $F$ are algebraically maximal Kaplansky fields and result follows from Fact \ref{F:final-extension}.
		Otherwise, $K$, $L$ and $F$ have characteristic $0$. By Proposition \ref{P:has-p-roots-and-more}(3) the respective relative algebraic closures of $K$ in $L$ and $F$ are also models of $T_1$. Denote them by $K^L$ and $K^F$, respectively.
		
		Let $\Deltap, v^\GDp$ and $v^\Deltap$ be as before. Since $vL=vK^L=vK=vK^F=vF$ we also get $v^\GDp L=v^\GDp K^L=v^\GDp K=v^\GDp K^L=v^\GDp F$. As $(K,v^\GDp)$ is a henselian field with residue characteristic $0$ we may use Lemma \ref{L:EnPr-Ramific}, and thus $K^L$ and $K^F$ are isomorphic as fields as long as $K^Lv^\GDp$ and $K^Fv^\GDp$ are isomorphic extension of $Kv^\GDp$. By Fact \ref{F:sep-closure} and Proposition \ref{P:has-p-roots-and-more}(3) we may apply the $p$-divisible case on $Lv^\GDp$ and $Fv^\GDp$, implying that they are, indeed, isomorphic.
	\end{proof}
	
	\subsubsection{The finite residue field case}
	Let $T_2$ be the theory of henselian $p$-valued fields of $p$-rank $d$. We consider it in the language of valued fields augmented by $d$ constants.  Notice that once we named constants, saying that they form an $\mathbb{F}_p$-basis for $\mathcal{O}_K/(p)$ is  a universal sentence. Hence every substructure of a model of $T_2$ is again a $p$-valued field of $p$-rank $d$.

	\begin{fact}\cite[Theorem 4.2]{raf}\label{F:T_2-elim-omegaDP}
	Every model of $T_2$ eliminates field quantifiers in the generalized Denef-Pas language.
	\end{fact}
		
	We review some facts concerning $p$-valued fields from \cite{PR84}.
	
	\begin{fact}\cite[Section 2.1]{PR84}
		$Kv$ and $[0,v(p)]$ are finite and \[d=\dim_{\mathbb{F}_{p}}Kv\cdot (|[0,v(p)]|-1).\]
	\end{fact}
	
	A \emph{$p$-adically closed field} is a $p$-valued field of $p$-rank d which does not admit any proper algebraic extension of the same $p$-rank. This is a first order property by \cite[Theorem 3.1]{PR84}.

	\begin{fact}\cite[Section 2.2, Theorem 3.1]{PR84}\label{F:core-field-p-adically-closed}
	For every henselian $p$-valued field $(K,v)$, the core field $(K_1,v^{\Delta_p})$ is a $p$-adically closed valued field of the same $p$-rank as $(K,v)$. 
	\end{fact}
	
	The following is well known.
	\begin{fact}\label{F:p-adically-elem-fiQp}
	Every $p$-adically closed field is elementary equivalent to a finite extension of the $p$-adics $\mathbb{Q}_p$.
	\end{fact}
	\begin{proof}
	Let $(K,v)$ be a $p$-adically closed field, which we may assume to be sufficiently saturated and thus to contain an isomorphic copy of $\mathbb{Q}_p$. Let $L$ be the intersection of $K$ and $\mathbb{Q}_p^{alg}$ (taken inside $K^{alg}$). It is also $p$-adically closed of the same $p$-rank as $K$ by \cite[Theorem 3.4]{PR84}. Thus $L$ is elementary equivalent to $K$ by model completeness \cite[Theorem 5.1]{PR84}. On the other hand, $L$ is a finite extension of $\mathbb{Q}_p$  since it has finite ramification index and finite inertia degree, see \cite[page 15]{PR84}.
	\end{proof}
	
%
%

	For future reference we sum up Fact \ref{F:char-strongdep} and Lemma \ref{L:infinite KV}:  
	\begin{lemma}\label{L: QEsNIP}
		If $(K,v)$ is a strongly dependent henselian field then $(K,v)\models T_1$ or $(K,v)\models T_2$. 
	\end{lemma}
	
	In fact, by \cite[Theorem 5.14]{HaH2017} we get the following: 
	\begin{corollary}
		Let $K$ be a strongly dependent field. Then $(K,v)\models T_1$ or $(K,v)\models T_2$ for any henselian valuation $v$ on $K$. 
	\end{corollary}

	\section{extending embeddings}
	The following results are proved in \cite{vd} for the $(0,0)$ case. We use results from \cite{FVK} to give the slight generalizations necessary for our needs.
	
	\begin{lemma}\label{L:val-remains}\cite[Lemma 5.20]{vd}
		Let $L,F$ be $ac$-valued fields, $f:L\to F$ an isomorphism of $ac$-valued fields $g:L'\to F'$ a valued-field isomorphism extending $f$ to some $ac$-valued field extensions  $L'/L$ and $F'/F$ such that $vL=vL'$. Then $g$  commutes with the ac-map.
	\end{lemma}
	\begin{proof}
		Let $x\in L'$, we may write $x=x_1x_2$ where $x_1\in L$ and $x_2\in \mathcal{O}_{L'}^\times$ and thus $ac(x)=ac(x_1)res(x_2)$ and \[f(ac(x))=f(ac(x_1))f(res(x_2))=ac(f(x_1))res(f(x_2))=ac(f(x)).\]
	\end{proof}

	If $(K,v)$ is a valued field we say, following F.-V. Kuhlmnann, that $K$ is Artin-Schreier closed if every irreducible polynomial of the form $x^p-x-c$ has a root in $K$ for $p=\mathrm{char} (Kv)>0$.
	
	\begin{lemma}\label{L:extending value}
		Let $L_1$ and $L_2$ be henselian ac-valued fields, $K$ a common henselian ac-valued subfield and $f:K\to L_2$ be the embedding (so $f$ commutes with the ac-map). Further assume that if $\mathrm{char}(Kv)=p>0$ then for $i=1,2$, for every $\gamma\in vL_i\setminus vK$ with $p\gamma\in vK$, there exists $a\in L_i$ with $a^p\in K$ and $va^p=p\gamma$.
		
		If $vL_1\equiv_{vK}vL_2$ then for any $\gamma_i\in vL_i$  with $tp(\gamma_1/vK)=tp(\gamma_2/vK)$ there exist $b_i\in L_i$ with $v(b_i)=\gamma_i$ such that $f$ may be lifted to an isomorphism of $ac$-valued fields $\tilde f:K(b_1)\to K(b_2)$, in particular $K(b_i)$ are ac-valued fields.
	\end{lemma}
	\begin{proof}
		
		If $\gamma_i\in vK$, we have nothing to prove. So we assume this is not the case. By abuse of notation we will not distinguish between $f$ and its extension $\tilde f$. We first show the following:
		\begin{enumerate}
			\item[Case 1:] Assume that $n\gamma_i\notin vK$ for every $n\geq 1$ and let $b_i\in L_i$ be any elements with $v(b_i)=\gamma_i$. 
			Replacing $b_i$, if needed, by $b_i/u_i$ with $u_i\in \mathcal{O}_{L_i}^\times$ such that $ac(u_i)=ac(b_i)$, we have $ac(b_i)=1$.
			
			Because $\gamma$ is not in the divisible hull of $vK$, necessarily $b_i$ is transcendental over $K$. By \cite[Lemma 3.23]{vd} (or \cite[Lemma 6.35]{FVK}), $v$ extends uniquely to $K(b_i)$, $K(b_i)v=Kv$ and $vK(b_i)=vK\oplus \mathbb{Z}\gamma_i$. 
			
			It remains to show that this extension commutes with the ac-map.
			Let $a\in K(b_1)$. Because $vK(b_1)=vK\oplus \mathbb{Z}\gamma_1$ there exists $c$ with $v(c)\in vK$ such that $v(a)=v(cb_1^n)$ for some $n\in\mathbb{Z}$. Thus $a=(cb_1^n)\cdot a/(cb_1^n)$ and $v(a/(cb_1^n))=0$. So $ac(a)=ac(cb_1^n)res(a/(cb_1^n))=ac(c)res(a/(cb_1^n))$ and 
			\[f(ac(a))=f(ac(c))f(res(a/(cb_1^n)))=\]\[ac(f(c)f(b_1^n)f(a/(cb_1^n))))=ac(f(a)).\]
			
			\item[Case 2:] Assume that $q\gamma_i\in vK$ for some prime $q$. 
			\begin{enumerate}
				\item[Case 2.1:] Assume that $q\neq \mathrm{char}(Kv)$. 
				Let $d\in K$ with $v(d)=q\gamma$. We may replace $d$ by $d/u$ with $u\in\mathcal{O}_K^\times$ such that $ac(u)=ac(d)$ and then $ac(d)=1$.
				Let $a_i\in L_i$ with $v(a_i)=\gamma_i$ and $ac(a_i)=1$. Then the polynomial $P_i(x)=x^q-d/a_i^q$ is over $\mathcal{O}_{L_i}$ and satisfies $v(P(1))>0$ and $v(P'(1))=0$. Indeed, $v(d/a_i^q)=0$ and $res(d/a_i^q)=ac(d)/ac(a_i^q)=\bar{1}$, so $\bar{P}(\bar{1})=0$. Since $q\neq \mathrm{char}(Kv)$ we automatically get that $\bar P$ is separable, so $\bar P'(1)\neq 0$.  
				
				This gives $u_i\in L_i$ such that $P_i(u_i)=0$ and $\bar{u_i}=1$. Now let $b_i=a_iu_i$, clearly $b_i^q=d$ and $ac(b_i)=1$. By \cite[Lemma 6.40]{FVK}, $K(b_i)v=Kv$ and $vK(b_i)=vK\oplus \mathbb{Z}\gamma_i$. Now $f|_{K(b_i)}$ commutes with the ac-map just as in case 1.
				
				\item[Case 2.2:] Assume that $q=p=\mathrm{char}(vK)>0$. 
				
				Just as in Case 2.1 we may use \cite[Lemma 6.40]{FVK} to extend the isomorphism (see the statement of \cite[Lemma 6.40]{FVK}).
			\end{enumerate}

		\end{enumerate}
		To conclude, if $n\gamma_i\in vK$ for some $n$ then, if $n$ is prime we are done using Case 2. Otherwise, fix some prime $q|n$ and it is now sufficient to  $\frac{n}{q}v_i$. Now the extension satisfies the assumptions of the lemma (this is obvious if $q\neq p$ and otherwise use Proposition \ref{P:has-p-roots-and-more}). We may now proceed by induction on $n$.
	\end{proof}

	\begin{lemma}\label{L:extending residue}
		Let $L_1$ and $L_2$ be henselian ac-valued field, $K$  a common henselian ac-valued subfield and $f:K\to L_2$ the embedding (so $f$ commutes with the ac-map). 
		
		Assume that $L_1v\equiv_{Kv}L_2v$ and let $\alpha_i\in L_iv$, with $\alpha_i$ separable over $Kv$ if it is algebraic, and $tp(\alpha_1/Kv)=tp(\alpha_2/Kv)$. Then there exist $a_i\in L_i$ with $\bar{a}_i=\alpha_i$ such that $f$ may be lifted to an isomorphism $f:K(a_1)\to K(a_2)$ of ac-valued fields. In particular $K(a_i)$ are ac-valued fields.
	\end{lemma}
	\begin{proof}
		It will suffice, by Lemma \ref{L:val-remains}, to show that we can extend $f$ to an isomorphism of valued field extensions preserving the value group of $K$. We break into two cases.
		\begin{enumerate}
			\item[Case 1:] Assume that $\alpha_i\in L_iv$ are transcendental over $Kv$. Pick any $a_i\in \mathcal{O}_{L_i}$ such that $\bar{a}_i=\alpha_i$. They are necessarily transcendental over $K$. Either by \cite[Lemma 6.35]{FVK} or by \cite[Lemma 3.22]{vd}, $v$ extends uniquely to $K(a_i)$ and $vK(a_i)=vK$. 
			
			\item[Case 2:] If $\alpha_i$ is algerbaic, let $P(X)\in\mathcal{O}_K[x]$ be monic such that $\bar{P}(x)$ is the minimal polynomial of $\alpha_i$ over $Kv$. Pick $b_i\in \mathcal{O}_L$ such that $\bar{b}_i=\alpha_i$. As the $\alpha_i$ are separable over $Kv$, $\bar{P}$ is separable and since $v(P(b_i))>0$ and $v(P'(b_i))=0$ we may find $a_i\in\mathcal{O}_L$ such that $P(a_i)=0$ and $\bar{a_i}=\alpha$. Either by \cite[Lemma 6.41]{FVK} or \cite[Lemma 3.21]{vd}, $v$ extends uniquely to $K(a_i)$ and $vK(a_i)=vK$. 
		\end{enumerate}
	\end{proof}
	
	%
	
	\section{eliminating field quantifiers}
	In this section all fields are assumed to be ac-valued considered in the Denef-Pas language\footnote{If they are models of $T_2$, we consider the structures in the language augmented by constants for an $\mathbb F_p$-basis of $\mathcal O_K/(p)$.}. We use the results from Section \ref{ss:2theories} and meld them with the proof from \cite[Section 3]{amc-kuh}. 
	
	Recall that a \emph{valuation transcendence basis} $\mathcal{T}$ for an extension $L/K$ is a transcendence basis for $L/K$ of the form \[\mathcal{T}=\{x_i,y_i:i\in I, j\in J\}\] such that $\{vx_i: i\in I\}$,  forms a maximal system of values in $vL$ which are $\mathbb Q$-linearly independent over $vK$ and the residues $\{\bar{y}_j:j\in J\}$, form a transcendence basis of $Lv/Kv$.
	
	Recall also that an algebraic extension of henselian fields $L/K$ is \emph{tame} if for every finite subextension $K'/K$:
	\begin{enumerate}
		\item $K'v/Kv$ is separable.
		\item if $p=\mathrm{char}(Kv)>0$ then $(vK':vK)$ is prime to $p$.
		\item $K'/K$ is a defectless extension.
	\end{enumerate}
	$K$ will be \emph{called} tame if it is henselian and every algebraic extension is a tame field. 
	
	\begin{lemma}\label{L:emb-for-tame}
		Let $L$ and $F$ be henselian ac-valued fields and $K$ a common ac-valued subfield. Assume that $L$ is a tame algebraic extension of the henselization of $K$.
		Then for every embedding $\tau=(\tau_\Gamma,\tau_k):(vL, Lv)\hookrightarrow (vF,Fv)$ over $K$, there is an embedding of $L$ in $F$ over $K$ inducing $\tau$ and commuting  with the ac-map.
	\end{lemma}
	\begin{proof}
		By the uniqueness of the henselization of $K$, we may extend the embedding $K\hookrightarrow F$  to the henselization of $K$. Since the henselization is an immediate extension, by Lemma \ref{L:val-remains} the embedding respects the ac-map. Since every algebraic extension is the union of its finite sub-extensions we may further assume that $L/K$ is finite.
		
		Since $Lv/Kv$ is finite and separable it is simple (recall $L/K$ is tame). Assume that $Kv(\alpha)=Lv$. By Lemma \ref{L:extending residue}(Case 2) there exists $c\in L$ with $\bar{c}=\alpha$, $K(c)v=Lv$ and $vK(c)=vK$ such that we may extend the embedding of $K\hookrightarrow F$ to an embedding $K(c)\hookrightarrow F$ respecting the ac-map.
		
		Set $K':=K(c)$. Since $L/K'$ is finite, the group $vL/vK'$ is a finite torsion group: \[vL/vK'=(\gamma_1 +vK)\mathbb{Z}\oplus\dots\oplus (\gamma_r +vK)\mathbb{Z},\] and by tameness the order of each $\gamma_i$ is prime to $p=\mathrm{char}(Kv)$. 
		Using Lemma \ref{L:extending value}(Case 2.1) repeatedly there exist $d_1,\dots,d_r\in L$ with $v(d_i)=\gamma_i$ such that $L/K'(d_1,\dots,d_r)$ is an immediate extension and we may extend the embedding of $K'\hookrightarrow F$ to an embedding $K'(d_1,\dots,d_r)\hookrightarrow F$ preserving the ac-map.
		
		Finally, $K'(d_1,\dots,d_r)$ is a tame field since it is an algebraic extension of a tame field, thus $L/K'(d_1,\dots,d_r)$ is defectless and immediate. Since henselian defectless fields are algebraically maximal, it follows that $K'(d_1,\dots,d_r)=L$ and we are done.
	\end{proof}

	\begin{lemma}\label{L:emb-tame-of-transc}
		Let $L$ and $F$ be henselian ac-valued fields and $K$ a common ac-valued subfield. Assume that $L$ admits a valuation transcendence basis $\mathcal{T}$ such that $L$ is a tame extension of $K(\mathcal{T})^h$. Then for every embedding $\tau=(\tau_\Gamma,\tau_k):(vL, Lv)\hookrightarrow (vF,Fv)$ over $K$, there is an embedding of ac-valued fields $L\hookrightarrow F$ over $K$ inducing  $\tau$. 
	\end{lemma}
	\begin{proof}
		By using Lemma \ref{L:extending value}(Case 1) and  Lemma \ref{L:extending residue}(Case 1) repeatedly we may extend the embedding $K\hookrightarrow F$ to an embedding $K(\mathcal{T}) \hookrightarrow F$ over $K$ respecting the ac-map, and inducing the embedding $\tau$. The result now follows from Lemma \ref{L:emb-for-tame}
	\end{proof}

	The following embedding theorem is, as usual, the main result: 
	
	\begin{theorem}\label{P:type-A}
		Let $L$ and $F$ be ac-valued fields. Assume that $F$ is $|L|^+$-saturated and $K$ a common ac-valued substructure in the Denef-Pas language. 
		If both $L$ and $F$ are models of $T_1$, then for every embedding $\tau=(\tau_\Gamma,\tau_k):(vL, Lv)\hookrightarrow (vF,Fv)$ over $K$, there is an embedding of $L$ in $F$ over $K$ inducing $\tau$ and preserving the ac-map.
	\end{theorem}
	\begin{proof}
		In order to embed $L$ in $F$ (over $K$) it suffices to embed every finitely generated sub-extension. So we may assume that $L/K$ is a finitely generated extension. 
		
		Let $\mathcal{T}=\{x_i,y_i: i\in I, j\in J\}$ be such that $\{vx_i: i\in I\}$ is a maximal system of $\mathbb Q$-linearly independent values in $vL$ over $vK$ and the residues $\{\bar{y}_j: j\in J\}$ are a transcendence basis of $Lv/Kv$. By \cite[Lemma 6.30]{FVK} $\mathcal{T}$ is algebraically independent over $K$.
		
		Let $L'$ be the maximal tame algebraic extension of the henselization $K(\mathcal{T})^h$ in $L$. By definition it is an algebraic extension so $\mathcal{T}$ is a transcendence basis for $L'/K$ and it is a tame extension of $K(\mathcal{T})^h$. We may use Lemma \ref{L:emb-tame-of-transc} and embed $L'$ in $F$ over $K$, this embedding induces $\tau$ and commutes with the ac-map.
		
		Since $L'$ is the maximal tame algebraic extension of $K(\mathcal{T})^h$, necessarily $vL/vL'$ is a $p$-group and $Lv/L'v$ is a purely inseparable algebraic extension. 		
		
		By Lemma \ref{P:has-p-roots-and-more}(2), we may repeatedly apply Lemma \ref{L:extending value}(Case 2.2) and extend the embedding of $L'\hookrightarrow F$ to an intermediary ac-valued field $L'\subseteq L''\subseteq L$ with $vL''=vL$ (since $vL/vL'$ is a $p$-group) and $Lv/L''v$ purely inseparable algebraic extension, in such a way that it commutes with the ac-map.
		

		By Lemma \ref{L:adap-johnson}, we may extend the embedding to $L'''$, the relative algebraic closure of $L''$ in $L$. The embedding $L'''\hookrightarrow F$ preserves the ac-map, by Lemma \ref{L:val-remains} since $vL''=vL'''$.

		By Proposition \ref{P:has-p-roots-and-more} $L'''$ is an ac-valued field which is a model of $T_1$, and hence it is algebraically maximal. Moreover,  notice that $L/L'''$ is immediate. 
		
	Let $a\in L\setminus L'''$. Since $L'''(a)/L'''$ is an immediate extension, by \cite[Theorem 1]{kaplansky} there exists a pseudo-cauchy sequence with $a$ as a pseudo-limit but with no pseudo-limit in $L'''$. Since $L'''$ is algebraically maximal it must be of transcendental type. Since $F$ is $|L|^+$-saturated it is also $|L|^+$-pc-complete. By \cite[Theorem 2]{kaplansky}, we may thus find a pc-limit $a'\in F$ and the map $a\mapsto a'$ extends the embedding to an embedding of valued fields, it preserves the ac-map by Lemma \ref{L:val-remains}. Doing this repeatedly we may embed $L$ in $F$ over $K$ as ac-valued fields.
	\end{proof}
	
	\begin{corollary}\label{C:field-elim-T1T2}
		Let $L$ and $F$ be ac-valued fields and $K$ a common substructure in the Denef-Pas language. If both $L$ and $F$ are models of $T_1$ then  \[(vL,Lv)\equiv_{(vK,Kv)}(vF,Fv) \Longrightarrow L\equiv_K F.\]
		
		As a result, $T_1$ eliminate field quantifiers and the value group and residue field are stably embedded as pure structures.
	\end{corollary}
	\begin{proof}
		There exist elementary extensions $L^*$ and $F^*$ of $L$ and $F$, respectively, such that $vL^*\cong_{vK}vF^*$ and $L^*v\cong_{Kv}F^*v$. Since $L^*\equiv_{K} F^*$ implies $L\equiv_{K} F$, we may assume from the start that $\tau=(\tau_\Gamma,\tau_k):(vL,Lv)\cong_{(vK,Kv)}(vF,Fv)$. The rest is standard and follows from the previous theorem.
	\end{proof}
	
	As special cases of the above corollary consider $K=(\mathbb Q, \mathbb Q, 0)$ if $L$ is of equi-characteristic $(0,0)$; let $K=(\mathbb F_p,\mathbb F_p, 0)$  if $L$ is of equi-characteristic $(p,p)$ (for $p>0$) and in mixed characteristic take $K=(\mathbb Q, \mathbb F_p, v(p)\mathbb Z)$ allowing us to conclude: 
	\begin{corollary}\label{C: completeness}
		The theory $T_1$ is complete once: 
		\begin{enumerate}
			\item the complete theories of the value group and the residue field are fixed and
			\item $v(p)$ is specified,
		\end{enumerate}  
	\end{corollary}

	Every algebraically maximal Kaplansky field is a model of $T_1$ so
	\begin{corollary}
		The theory of any algebraically maximal Kaplansky ac-valued field eliminates field quantifiers in the Denef-Pas language.
	\end{corollary}
	
	
	Combining these results with Lemma \ref{L: QEsNIP}, Fact \ref{F:T_2-elim-omegaDP}, and the fact that it is known for the $(0,0)$ case, we have shown: 
	\begin{corollary}\label{C:elim-fqe-strong}
		Let $(K,v)$ be a strongly dependent henselian ac-valued field. If $Kv$ is infinite then $\mathrm{Th}(K,v)$ eliminates field quantifiers in the Denef-Pas language, and if $Kv$ is finite then $\mathrm{Th}(K,v)$ eliminates field quantifiers in the generalized Denef-Pas language.
	\end{corollary}
	
	The above corollaries give Theorem \ref{T:theorem A} of the introduction. We now proceed to some applications. We remind (\cite[Proposition 3.4]{HaH2017}) that strongly dependent ordered abelian groups have quantifier elimination in the language \[L=L_{oag}\cup \{(x=_{H_i} y +k_{G/H_i})_{k\in \mathbb{Z}, i< \alpha}, (x\equiv_{m,H_i} y+k_{G/H_i})_{k\in \mathbb{Z},m\in\mathbb{N}, i< \alpha}\},\] where
	\begin{list}{•}{}
		\item $L_{oag}$ is the language of ordered groups, 
		\item for each $k\in \mathbb{Z}$, "$x=_H y+k_{G/H}$" is defined by $\pi (x)=\pi (y)+k_{G/h}$ for $\pi: G\to G/H$ and $k_{G/H}$ denotes $k$ times the minimal positive element of $G/H$, if it exists, and $0$ otherwise.
		\item for each $k\in\mathbb{Z}$ and each $m\in \mathbb{N}$, "$x\equiv_{m,H} y+k_{G/H}$" is defined by $\pi(x)\equiv_m \pi(y) +k_{G/H}$.
	\end{list}
	
	Thus, the above corollary implies that a strongly dependent henselian field $(K,v)$ has quantifier elimination modulo $Kv$ in the (generalized) Denef-Pas language augmented by the new predicates in $L$. In particular, if $Kv$ has explicit quantifier elimination in some natural language expanding the language of rings (e.g., $Kv$ is an algebraically closed field or a real closed field) we get complete quantifier elimination for $(K,v)$. This strengthens \cite[Theorem 3.2.16]{johnson}.
	
	%
	%
	
	Strong dependence of an ac-valued field $(K,v)$ admitting elimination of field quantifiers in the Denef-Pas language follows from strong dependence of $vK$ and $Kv$ by \cite[Claim 1.17(2)]{ShDep09}. This result can probably be proved for ac$_\omega$-valued fields as well, but we take a slightly different approach. We note that \cite[Claim 1.17(2)]{ShDep09} gives this result for a slightly different language and that in $\mathbb{Q}_p$ Shelah's $3$-sorted language is bi-interpretable with the multi-sorted structure given by the generalized Denef-Pas language. This will suffice for our needs.
	
	\begin{lemma}\label{L:transfer for models of T1T2}
		Let $(K,v)$ be either a model of $T_1$ or of $T_2$. If $Kv$ and $vK$ are strongly dependent then so is $(K,v)$.
	\end{lemma}	 
	\begin{proof}
	If $(K,v)$ is a model of $T_1$ then by passing to an elementary extension we may assume that $(K,v)$ is an ac-valued field and thus, by Corollary \ref{C:field-elim-T1T2}, eliminates field quantifiers in the Denef-Pas language. We may thus use \cite[Claim 1.17(2)]{ShDep09}.
	
	Assume $(K,v)$ is a model of $T_2$.  By Fact \ref{F:core-field-p-adically-closed} and Fact \ref{F:p-adically-elem-fiQp}, $(K_1,v^{\Delta_p})$,its core field, is elementary equivalent to a finite extension of $\mathbb{Q}_p$ and thus strongly dependent by \cite[Claim 1.15, 1.16 and Observation 1.19(1)]{ShDep09}.
	
	Now consider the valued field $(K,v^{vK/\Delta_p})$. Its residue field is $K_1$ and hence strongly dependent and its value groups is a quotient of a strongly dependent abelian group and hence also strongly dependent by \cite[Theorem 4.20]{HaH2017} (see also \cite[Corollary 4.21]{HaH2017}). By passing to an elementary extension we may assume that $(K,v^{vK/\Delta_p})$ is an ac-valued field, and since it is of characteristic $(0,0)$ it admits elimination of field quantifiers and thus strongly dependent by \cite[Claim 1.17(2)]{ShDep09}. Since strong dependence is preserved under reducts and elementary equivalence, we get that $K$ is strongly dependent. The result now follows by \cite[Theorem 5.14]{HaH2017} applied to $K$ and the henselian valuation $v$.
	\end{proof}
	
	%
	%
	%
	%
	%
	
	We end by showing that elimination of field quantifiers of the henselization can be deduced from strong dependence of the valued field.
	
	\begin{proposition}\label{P:henselisation}
		Let $(K,v)$ be a strongly dependent valued field. Then its henselization $(K^h,v)$ is also strongly dependent.
	\end{proposition}
	\begin{proof}
		
		First we show that $(K^h,v)$ is either a model of $T_1$ or of $T_2$.
		By \cite[Theorem 4.2.2]{johnson}, $(K,v)$ is defectless and hence $(K^h,v)$ is algebraically maximal. We break into cases:
		
		\underline{If $(\mathrm{char}(K),\mathrm{char}(Kv))=(0,0)$} then $(K^h,v)$ is a model of $T_1$.
		
		\underline{If $(\mathrm{char}(K),\mathrm{char}(Kv))=(p,p)$} then by \cite[Theorem 4.3.1]{johnson} $Kv$ is infinite. $Kv$ is perfect by strong dependence and $p$ does not divide the degree of any finite extension of $Kv$. Perfection of $K$ implies that $vK$ is $p$-divisible. Thus $(K,v)$ is Kaplansky and, since the henselization is an immediate extension, $(K^h,v)$ is algebraically maximal Kaplansky, so a model of $T_1$.
		
		\underline{Assume $(\mathrm{char}(K),\mathrm{char}(Kv))=(0,p)$}. If $Kv$ is finite then by \cite[Theorem 4.3.1]{johnson}, $[0,v(p)]$ is finite so $(K^h,v)$ is a model of $T_2$. If $Kv$ is infinite, as before, $(K^h,v)$ is a model of $T_1$.
		
		Since the henselization $(K^h,v)$ is an immediate extension, $vK^h$ and $K^hv$ are strongly dependent. By Lemma \ref{L:transfer for models of T1T2}, $(K^h,v)$ is strongly dependent.
	\end{proof}

	A similar proof gives the following:
	
	\begin{proposition}
		Let $(K,v)$ be strongly dependent such that $(K^h,v)\models T_1$. Then the inertia field $(K^t,v)$ of $(K,v)$ is strongly dependent. 
	\end{proposition}
	\begin{proof}
		The definition and basic properties of the inertia field can be found in e.g. \cite[Section 7.4]{FVK}. Since $(K^h,v)\models T_1$ also $(K^t,v)$ is henselian and defectless. Moreover, $K^t$ and $K^tv$ are perfect. Since $vK^t=vK^h$ in order to show that that $(K^t,v)\models T_1$ it remains to show that the degree of every finite extension of $K^tv$ is prime to $p$, but $K^tv$ being separably closed and perfect it is algebraically closed, so there is nothing to prove.
		
		The rest is as in Proposition \ref{P:henselisation}.	
	\end{proof}
	
	This proves Theorem \ref{T:theorem B} of the introduction.
	
	\begin{remark}
		Notice that a similar proof shows that if $(K^h,v)\models T_1$ then $(L,v)$ admits elimination of field quantifiers whenever $K^h\subseteq L\subseteq K^t$. 
	\end{remark}

	
	The following example shows that the requirement that $(K^h,v)\models T_1$ is necessary.
	
	\begin{example}
		Let $(K,v)$ be a strongly dependent field with discrete value group and finite residue field. Then, by Fact \ref{F:char-strongdep} $(K^t,v)$ is not strongly dependent, despite the fact that $vK^t=vK$ is strongly dependent and $K^tv$ is algebraically closed (being the algebraic closure of $\mathbb F_p$). Note that $(K^t,v)$ is, in addition, henselian and defectless (being an algerbaic extension of $(K^h,v)$ which is strongly dependent), so algebraically maximal, with algebraically closed residue field. 	
	\end{example}
	
	We point out the following result: 
	
	\begin{proposition}
		Let $(K,v)\equiv(L,w)$ (in the three-sorted language) be strongly dependent such that $(K^h,v), (L^h,w)\models T_1$. Then $(L^h,w)\equiv (K^h,v)$ and $(L^t,w)\equiv (K^t,v)$. 
	\end{proposition}
	\begin{proof}
		By Corollary \ref{C: completeness}, the assumption that $(K,v)\equiv(L,w)$ and the fact that the henselisation is an immediate extension, imply that $v(p)=w(p)$ for $p=\mathrm{char} (Lw)=\mathrm{char} (Kv)$, and therefore $(K^h,v)\equiv (L^h,w)$. For the intertia fields, recall that if two fields are elementary equivalent then so are their algebraic closures.
	\end{proof}
	
	%
	%

	\section{Geometric Fields}\label{s:geometric}
	In this final section, we use arguments from \cite[Theorem 5.5]{Junker2010}  for pure henselian valued fields of characteristic $0$, to show that henselian fields of finite dp-rank are geometric fields (see below for the definition). The proof in \cite{Junker2010} is duplicated almost verbatim, we give the proof for the sake of completeness.
	
	\begin{proposition}\label{P:acl-alg}
		Let $K$ be a strongly dependent field and $v$ a non-trivial henselian valuation on $K$. Then for every $K\equiv K'$ (in the language of rings) and $A\subseteq K'$, $acl_{K'}(A)=\mathbb{F}_0( A) ^{alg}\cap K'$, where $\mathbb{F}_0$ is the prime field of $K'$.
	\end{proposition}
	\begin{remark}
		A field satisfying this proposition in called \emph{very slim} in \cite{Junker2010}.
	\end{remark}
	\begin{proof}
		If $K$ is separably closed, and hence --  by perfection -- algebraically closed, this is known (and follows, essentially, from quantifier elimination in ACVF). 
		
		Otherwise, by \cite[Remark 7.11]{PrZ1978} $K'$ also admits a definable henselian topology. By \cite[Lemma 4.11]{Junker2010}, it is enough to prove the statement for an elementary extension of $K'$. We, thus, assume that $K'$ is $\aleph_1$-saturated. By \cite[Theorem 7.2]{PrZ1978}, $K'$ admits a non-trivial henselian valuation which we will also denote by $v$. If $(K',v)$ is of mixed characteristic then, as in the proof of Lemma \ref{L:transfer for models of T1T2},   by passing to a coarsening  we may assume that $(K',v)$ is a model of $T_1$. By Fact \ref{F:exist-ac} $(K',v)$ admits an angular component map. 
		
		Let $\varphi(x)$ be an $A$-definable algebraic formula and assume there exists $a$ satisfying $\varphi(x)$ which is transcendental over $\mathbb{F}_0(A)$. 
		
		\begin{claim}
			For every non constant polynomial $p(x)$ over $A$ such that $p(a)\neq 0$, there exists a formula $\psi_p(x)$ over $A$ satisfying:
			\begin{enumerate}
				\item $\psi_p$ is not algebraic.
				\item $a$ satisfies $\psi_p$ and for every $a'\models \psi_p$:
				\[rv(p(a))=rv(p(a')),\]
				where $rv:(K')^\times\to (K')^\times/(1+\mathcal{M})$ is the natural projection ($\mathcal{M}$ is the maximal ideal corresponding to the valuation $v$).
			\end{enumerate}
		\end{claim}
		\begin{claimproof}
			For any $x$ note that $rv(p(a))=rv(p(x))$ if and only if $v(p(x)-p(a))>v(p(a))$. As $p(a)\neq 0$ this is equivalent to $v(p(x)/p(a)-1)>0$. Let $\psi_p(x)$ be this latter formula.
		\end{claimproof}
		
		Since $rv(x)=rv(y)$ implies $v(x)=v(y)$ and $ac(x)=ac(y)$, by elimination of field quantifiers (Corollary \ref{C:elim-fqe-strong}), see also \cite[Theorem 5.5]{Junker2010}, and using the claim we may find a non algebraic formula $\psi(x)$ over $A$ such that for every $a'$ satisfying $\psi$, $a'$ satisfies $\phi$ as well. Contradicting the fact that $\varphi(x)$ is algebraic.
	\end{proof}
	
	\begin{corollary}\cite[Proof of Corollary 5.6]{Junker2010}
		A strongly dependent henselian field has no proper infinite ring-definable subfield.
	\end{corollary}

	In \cite[Remark 2.10]{HrP1994}, Hrushovski-Pillay define the notion of a \emph{geometric field}, it is a field $K$ satisfying:
	\begin{list}{•}{}
		\item $K$ is perfect,
		\item For every $K'\equiv K$ (in the language of rings) and $A\subseteq K'$, \[acl_{K'}(A)=\mathbb{F}_0(A)^{alg},\] where $\mathbb{F}_0$ is the prime field,
		\item Every $K'\equiv K$ eliminates $\exists^\infty$.
	\end{list}
	
	Those fields, when sufficiently saturated, enjoy a nice group configuration theorem, see \cite[Section 3]{HrP1994}.
	
	Since every strongly dependent field is perfect, combined with Proposition \ref{P:acl-alg} and elimination of $\exists^\infty$ in the finite dp-rank case (\cite[Corollary 2.2]{DoG2017}) we get:
	
	\begin{proposition}
		Every henselian field of finite dp-rank is a geometric field.
	\end{proposition}
	
	Recall the following:
	\begin{definition}
		A structure $M$ has \emph{geometric elimination of imaginaries} if every $b\in M^{eq}$ is interalgebraic with some finite tuple from $M$.
	\end{definition}
	
	\begin{proposition}
		Let $K$ be a strongly dependent field admitting a non-trivial henselian valuation $v$. If $K$ does not admit any non-trivial definable valuation then $K$ is either algebraically closed or real closed.
		
		As a result, if $K$ has geometric elimination of imaginaries, or more specifically if it is \emph{surgical} (condition $(E)$ in \cite[Definition 2.4]{HrP1994}) then it is either algebraically closed or real closed.
	\end{proposition}
	\begin{proof}
		By strong dependence $K$ is perfect, and by \cite[Proposition 5.2]{HaH2017}  the residue field $Kv$ is also perfect.
		By \cite[Proposition 5.5]{HaH2017} the value group $vK$ is divisible. By the proof of \cite[Proposition 2.4]{Jah2016}, $Kv$ is either real closed or algebraically closed (the proof shows that unless $Kv$ is real closed or separably closed there exists a non-trivial definable valuation on $K$).
		
		By \cite[Corollary 5.15]{HaH2017}, $(K,v)$ is algebraically maximal, thus if $Kv$ is algebraically closed then, since $vK$ is divisible, so is $K$.
		Otherwise, $Kv$ is real closed and thus necessarily so is $K$.
		
		As for the last statement, first note that by Proposition \ref{P:acl-alg} the (model theoretic) $acl(\cdot)$-operator on $K$ satisfies the Steinitz exchange (because it coincides with field theoretic algebraic closure in $K$) and therefore its theory is pregeoemtric (\cite[Definition 2.1]{Gag2005}). It follows, \cite[Corollary 3.6]{Gag2005}, that if $\mathrm{Th}(K)$ has geometric elimination of imaginaries it is surgical. This means, in particular, that there cannot be a $K$-definable equivalence relation on the field with infinitely many infinite classes. But had $K$ admitted a non-trivial definable valuation $u$ the formula $u(x)=u(y)$ would be such an equivalence relation. Thus, by the first part of the proposition, $K$ must be either real closed or algebraically closed. 
	\end{proof}
	
	We note that the first part of the above proposition is also true, more generally, in the strictly dependent case by  \cite[Corollary 1.3]{DuHaKu}
	
	\appendix
	
	\section{Elimination of Field Quantifiers in the RV-Language}\label{A:appendix}
	We conclude by showing that strongly dependent fields (in fact, models of either $T_1$ or $T_2$) eliminate field quantifiers in the RV-language. As the proof is, essentially, similar to what we have done, we only sketch the argument. Also, as it is already known for models of $T_2$ in a more general language than the one described below (see \cite{basarab} and \cite{amc-kuh}) we will focus on models of $T_1$. We briefly review some definitions, see also, e.g. \cite{flenner}.
	
	Let $(K,v)$ be a valued field. The group $1+\mathcal{M}_K$ of $1$-units is a subgroup of $K^\times$. Set $RV_K:=K^\times/(1+\mathcal{M}_K)$ and let \[rv_K:K^\times\to RV_K\] be the natural quotient homomorphism. We may extend this map to all of $K$ by adding a new symbol for $rv_K(0)$. Note that $(Kv)^\times$ embeds in $RV_K$ and we have the following exact sequence
	
	\[\xymatrix{
		1\ar[r] & (Kv)^\times\ar[r]^\iota & RV_K\ar[r]^{v}& vK\ar[r] & 0
	}.\]
	$RV_K$ also inherits an image of the addition from $K$ denoted by $\oplus$, see \cite{flenner} for more information. We consider $RV_K$ as a structure in the language $\{\cdot, ^{-1},\oplus,1,v\}$. The RV-structure for the valued field $K$ is a two sorted structure $(K,RV_K)$ together with the map $rv_K$. We need the following extension of Lemma \ref{L:extending value}:
	
	\begin{lemma}\label{L:extend p-root-rv}
		Let $(L_1,v)$ and $(L_2,v)$ be models of $T_1$ in the RV-language, and $K$ a common substructure with $\mathrm{char}(Kv)=p>0$ and $\tau: RV_{L_1}\hookrightarrow RV_{L_2}
		$ an embedding of the RV-sorts over $RV_K$.
		
		Assume that $vL_1/vK$ is a finitely generated $p$-group, then there exist $d_1,\dots,d_k\in L_2$ and an embedding of valued fields $K(d_1,\dots,d_k)\hookrightarrow L_2$ over $K$ such that 
		
		\begin{enumerate}
			\item $vK(d_1,\dots,d_k)=vL_1$, and
			\item the embedding of RV-structures \[(K(d_1,\dots,d_k), RV_{K(d_1,\dots,d_k)})\hookrightarrow (L_2,RV_{L_2})\] over $(K,RV_K)$ induced by the embedding of valued fields \[K(d_1,\dots,d_k)\hookrightarrow L_2,\] over $K$, lifts $\tau|_{K(d_1,\dots,d_k)}$.
		\end{enumerate}
	\end{lemma}
	\begin{proof}
		By induction, we may assume that $vL_1/vK$ is generated by one element, e.g. $vL_1=vK+\mathbb{Z}\alpha$ with $p\alpha\in vK$, $\alpha\notin vK$. By Proposition \ref{P:has-p-roots-and-more}(2) for $T_1$, there exists $d\in L_1$ satisfying $d^p\in K$ and $v(d^p)=p\alpha$. By \cite[Lemma 6.40]{FVK}, $v$ extends uniquely to $K(d)$ and \[vK(d)=vK+\mathbb{Z}\alpha \text{ and } K(d)v=Kv.\]

		Since $rv_K(d^p)=y^p$, has a solution in $RV_{L_1}$ (e.g. $y=rv_{L_1}(d)$), by the embedding $\tau$, there exists $e\in L_2$ such that \[rv_K(d^p)=rv_{L_2}(e^p).\] The problem is that we do not know that $e^p\in K$, but notice that $v(e^p)=v(d^p)\in pvL_2$. Hence, again, by Proposition \ref{P:has-p-roots-and-more}(2) for $T_1$, there exists $t\in L_2$ such that $t^p=d^p$.
		
		By uniqueness in \cite[Lemma 6.40]{FVK}, $K(d)\subseteq L_1$ and $K(t)\subseteq L_2$ are isomorphic as valued fields over $K$. It remains to show that this isomorphism is a lifting of $\tau|_{K(d)}$.
		
		\begin{claim}
			$\tau (rv_{L_1}(d))=rv_{L_2}(t)$
		\end{claim}
		\begin{claimproof}
			By the way we chose $e$, $\tau (rv_{L_1}(d))=rv_{L_2}(e)$ and thus also $v(d^p)=v(e^p)$ and $res(e^p/d^p)=1$ (\cite[Proposition 1.3.3]{flenner}). Since $v(e^p)=v(t^p)$, $v(e)=v(t)$, and since $res(e^p/t^p)=res(e^p/d^p)=1$ we have $res(e/t)^p=1$. But $\mathrm{char}(Kv)=p>0$ so $res(e/t)=1$. Again, by \cite[Proposition 1.3.3]{flenner}, $rv_{L_2}(e)=rv_{L_2}(t)$ and we are finished. 
			
		\end{claimproof}
		
		Let $x\in K(d)$.  By the above there exist $b\in K$ such that \[v(x^{-1}bd^n)=0.\] It follows from this and $K(d)v=Kv$ that there exists also $c\in K^\times$ with $v(c)=0$ such that $v(x^{-1}bd^nc)=0$ and $res(x^{-1}bd^nc)=1$. Hence, see for instance \cite[Proposition 1.3.3]{flenner},
		\[rv_{K(d)}(x)=rv_K(b)\cdot rv_{K(d)}(d)^n\cdot rv_K(c).\]
	\end{proof}
	
	\begin{theorem}
		Let $L$ and $F$ be models of $T_1$ in the RV-language, with $F$ $|L|^+$-saturated, and $K$ a common substructure. Then for every embedding $\tau:RV_L\hookrightarrow RV_F$ over $RV_K$, there is an embedding of $L$ in $F$ over $K$ inducing $\tau$.
	\end{theorem}
	\begin{proof}
		As in the proof of Theorem \ref{P:type-A}, let $L'$ be the maximal tame algebraic extension of $K(\mathcal{T})^h$, where $\mathcal{T}$ is as in the proof. By \cite[Lemma 3.7]{amc-kuh}, we may embed $L'$ in $F$ over $K$, this embedding induces $\tau$. Necessarily $vL/vL'$ is a $p$-group and $Lv/L'v$ is a purely inseparable algebraic extension.
		
		Using Lemma \ref{L:extend p-root-rv}, we may extend the embedding to an intermediary field $L'\subseteq L''\subseteq L$ with $vL''=vL$ and $Lv/L''v$ a purely inseparable algebraic extension. The rest of the proof is as in Theorem \ref{P:type-A}.
	\end{proof}
	
	\begin{corollary}
		$T_1$ and $T_2$, and hence any strongly dependent henselian valued fields, eliminate field quantifiers in the RV-language (for $T_2$ you need the generalized RV-language).
	\end{corollary}
	
	\bibliographystyle{plain}
	\bibliography{kaplansky2}

\begin{thebibliography}{10}

\bibitem{basarab}
\c~Serban~A. Basarab.
\newblock Relative elimination of quantifiers for {H}enselian valued fields.
\newblock {\em Ann. Pure Appl. Logic}, 53(1):51--74, 1991.

\bibitem{belair}
Luc B\'elair.
\newblock Types dans les corps valu\'es munis d'applications coefficients.
\newblock {\em Illinois J. Math.}, 43(2):410--425, 1999.

\bibitem{raf2}
Raf Cluckers and Immanuel Halupczok.
\newblock Integration of functions of motivic exponential class, uniform in all
  non-archimedean local fields of characteristic zero.
\newblock {\em J. \'{E}c. polytech. Math.}, 5:45--78, 2018.

\bibitem{raf}
Raf Cluckers, Leonard Lipshitz, and Zachary Robinson.
\newblock Analytic cell decomposition and analytic motivic integration.
\newblock {\em Ann. Sci. \'{E}cole Norm. Sup. (4)}, 39(4):535--568, 2006.

\bibitem{DoG2017}
Alfred Dolich and John Goodrick.
\newblock Strong theories of ordered {A}belian groups.
\newblock {\em Fundamenta Mathematicae}, 236(3):269--296, 2017.

\bibitem{DuHaKu}
K.~{Dupont}, A.~{Hasson}, and S.~{Kuhlmann}.
\newblock {Definable Valuations induced by multiplicative subgroups and NIP
  Fields}.
\newblock {\em ArXiv e-prints}, April 2017.

\bibitem{EnPr}
Antonio~J. Engler and Alexander Prestel.
\newblock {\em Valued fields}.
\newblock Springer Monographs in Mathematics. Springer-Verlag, Berlin, 2005.

\bibitem{flenner}
Joseph~Doyle Flenner.
\newblock {\em The relative structure of henselian valued fields}.
\newblock ProQuest LLC, Ann Arbor, MI, 2008.
\newblock Thesis (Ph.D.)--University of California, Berkeley.

\bibitem{Gag2005}
Jerry Gagelman.
\newblock Stability in geometric theories.
\newblock {\em Annals of Pure and Applied Logic}, 132(2-3):313--326, 2005.

\bibitem{HaH2017}
Yatir Halevi and Assaf Hasson.
\newblock Strongly dependent ordered abelian groups and henselian fields, 2017.
\newblock preprint, \url{https://arxiv.org/abs/1706.03376}.

\bibitem{HrP1994}
Ehud Hrushovski and Anand Pillay.
\newblock Groups definable in local fields and pseudo-finite fields.
\newblock {\em Israel Journal of Mathematics}, 85(1-3):203--262, 1994.

\bibitem{Jah2016}
Franziska Jahnke.
\newblock When does nip transfer from fields to henselian expansions?, 2016.
\newblock preprint, \url{https://arxiv.org/abs/1607.02953}.

\bibitem{johnson}
William~Andrew Johnson.
\newblock {\em Fun with Fields}.
\newblock PhD thesis, University of California, Berkeley, 2016.

\bibitem{Junker2010}
Markus Junker and Jochen Koenigsmann.
\newblock Schlanke {K}\"orper (slim fields).
\newblock {\em The Journal of Symbolic Logic}, 75(2):481--500, 2010.

\bibitem{KaScWa}
Itay Kaplan, Thomas Scanlon, and Frank~O. Wagner.
\newblock Artin-{S}chreier extensions in {NIP} and simple fields.
\newblock {\em Israel J. Math.}, 185:141--153, 2011.

\bibitem{kaplansky}
Irving Kaplansky.
\newblock Maximal fields with valuations.
\newblock {\em Duke Math. J.}, 9:303--321, 1942.

\bibitem{FVK}
Frantz-Victor Kuhlmann.
\newblock Valuation theory.
\newblock Available on http://math.usask.ca/\textasciitilde{}fvk/Fvkbook.htm,
  2011.

\bibitem{amc-kuh}
Franz-Viktor Kuhlmann.
\newblock Quantifier elimination for {H}enselian fields relative to additive
  and multiplicative congruences.
\newblock {\em Israel J. Math.}, 85(1-3):277--306, 1994.

\bibitem{pas90}
Johan Pas.
\newblock On the angular component map modulo {$P$}.
\newblock {\em J. Symbolic Logic}, 55(3):1125--1129, 1990.

\bibitem{PR84}
Alexander Prestel and Peter Roquette.
\newblock {\em Formally {$p$}-adic fields}, volume 1050 of {\em Lecture Notes
  in Mathematics}.
\newblock Springer-Verlag, Berlin, 1984.

\bibitem{PrZ1978}
Alexander Prestel and Martin Ziegler.
\newblock Model-theoretic methods in the theory of topological fields.
\newblock {\em Journal f\"ur die Reine und Angewandte Mathematik},
  299(300):318--341, 1978.

\bibitem{silvain}
Silvain Rideau.
\newblock Some properties of analytic difference valued fields.
\newblock {\em J. Inst. Math. Jussieu}, 16(3):447--499, 2017.

\bibitem{ShDep09}
Saharon Shelah.
\newblock Dependent first order theories, continued.
\newblock {\em Israel J. Math.}, 173:1--60, 2009.

\bibitem{vd}
Lou van~den Dries.
\newblock Lectures on the model theory of valued fields.
\newblock In {\em Model theory in algebra, analysis and arithmetic}, volume
  2111 of {\em Lecture Notes in Math.}, pages 55--157. Springer, Heidelberg,
  2014.

\end{thebibliography}
	
\end{document}